\documentclass[11.75pt]{amsart}
\usepackage{amsmath,amssymb,amsbsy,amsfonts,amsthm,latexsym,amsopn,amstext,
                                               amsxtra,euscript,amscd,color,bm}
                                                                                           
\usepackage[colorlinks,linkcolor=blue,anchorcolor=blue,citecolor=blue,backref=page]{hyperref}

\usepackage[right=2.5cm,left=2.5cm,top=2cm,bottom=2cm]{geometry}

\usepackage[np]{numprint}
\npdecimalsign{\ensuremath{.}}

\hypersetup{breaklinks=true}

\begin{document}

\def\ov#1{{\overline{#1}}} 
\def\un#1{{\underline{#1}}}
\def\wh#1{{\widehat{#1}}}
\def\wt#1{{\widetilde{#1}}}

\newcommand{\Ch}{{\operatorname{Ch}}}
\newcommand{\Elim}{{\operatorname{Elim}}}
\newcommand{\proj}{{\operatorname{proj}}}
\newcommand{\h}{{\operatorname{h}}}

\newcommand{\hh}{\mathrm{h}}
\newcommand{\aff}{\mathrm{aff}}
\newcommand{\Spec}{{\operatorname{Spec}}}
\newcommand{\Res}{{\operatorname{Res}}}
\newcommand{\Orb}{{\operatorname{Orb}}}

\renewcommand*{\backref}[1]{}
\renewcommand*{\backrefalt}[4]{%
    \ifcase #1 (Not cited.)%
    \or        (p.\,#2)%
    \else      (pp.\,#2)%
    \fi}
\def\lc{{\mathrm{lc}}}
\newcommand{\hcan}{{\operatorname{\wh h}}}

\newcommand{\hooklongrightarrow}{\lhook\joinrel\longrightarrow}

\newcommand{\bfa}{{\boldsymbol{a}}}
\newcommand{\bfb}{{\boldsymbol{b}}}
\newcommand{\bfc}{{\boldsymbol{c}}}
\newcommand{\bfd}{{\boldsymbol{d}}}
\newcommand{\bff}{{\boldsymbol{f}}}
\newcommand{\bfg}{{\boldsymbol{g}}}
\newcommand{\bfell}{{\boldsymbol{\ell}}}
\newcommand{\bfp}{{\boldsymbol{p}}}
\newcommand{\bfq}{{\boldsymbol{q}}}
\newcommand{\bfs}{{\boldsymbol{s}}}
\newcommand{\bft}{{\boldsymbol{t}}}
\newcommand{\bfu}{{\boldsymbol{u}}}
\newcommand{\bfv}{{\boldsymbol{v}}}
\newcommand{\bfw}{{\boldsymbol{w}}}
\newcommand{\bfx}{{\boldsymbol{x}}}
\newcommand{\bfy}{{\boldsymbol{y}}}
\newcommand{\bfz}{{\boldsymbol{z}}}

\newcommand{\bfA}{{\boldsymbol{A}}}
\newcommand{\bfF}{{\boldsymbol{F}}}
\newcommand{\bfG}{{\boldsymbol{G}}}
\newcommand{\bfR}{{\boldsymbol{R}}}
\newcommand{\bfQ}{{\boldsymbol{Q}}}
\newcommand{\bfT}{{\boldsymbol{T}}}
\newcommand{\bfU}{{\boldsymbol{U}}}
\newcommand{\bfX}{{\boldsymbol{X}}}
\newcommand{\bfY}{{\boldsymbol{Y}}}
\newcommand{\bfZ}{{\boldsymbol{Z}}}

\newcommand{\bfeta}{{\boldsymbol{\eta}}}
\newcommand{\bfxi}{{\boldsymbol{\xi}}}
\newcommand{\bfrho}{{\boldsymbol{\rho}}}

\newcommand{\PreP}{{\mathrm{PrePer}}}

\def\fM{{\mathfrak M}}
\def\fE{{\mathfrak E}}
\def\fF{{\mathfrak F}}



\newfont{\teneufm}{eufm10}
\newfont{\seveneufm}{eufm7}
\newfont{\fiveeufm}{eufm5}
%
%
\newfam\eufmfam
                \textfont\eufmfam=\teneufm \scriptfont\eufmfam=\seveneufm
                \scriptscriptfont\eufmfam=\fiveeufm
%
%
\def\frak#1{{\fam\eufmfam\relax#1}}
%

\def\ts{\thinspace}

\newtheorem{theorem}{Theorem}
\newtheorem{lemma}[theorem]{Lemma}
\newtheorem{claim}[theorem]{Claim}
\newtheorem{cor}[theorem]{Corollary}
\newtheorem{prop}[theorem]{Proposition}
\newtheorem{question}[theorem]{Open Question}

\newtheorem{rem}[theorem]{Remark}
\newtheorem{definition}[theorem]{Definition}

\newenvironment{dedication}
  {
   \thispagestyle{empty}
   \vspace*{\stretch{1}}
   \itshape             
   \raggedleft          
  }
  {\par 
   \vspace{\stretch{3}} 
  }

\numberwithin{table}{section}
\numberwithin{equation}{section}
\numberwithin{figure}{section}
\numberwithin{theorem}{section}


\def\squareforqed{\hbox{\rlap{$\sqcap$}$\sqcup$}}
\def\qed{\ifmmode\squareforqed\else{\unskip\nobreak\hfil
\penalty50\hskip1em\null\nobreak\hfil\squareforqed
\parfillskip=0pt\finalhyphendemerits=0\endgraf}\fi}

\def\fA{{\mathfrak A}}
\def\fB{{\mathfrak B}}

\def\cA{{\mathcal A}}
\def\cB{{\mathcal B}}
\def\cC{{\mathcal C}}
\def\cD{{\mathcal D}}
\def\cE{{\mathcal E}}
\def\cF{{\mathcal F}}
\def\cG{{\mathcal G}}
\def\cH{{\mathcal H}}
\def\cI{{\mathcal I}}
\def\cJ{{\mathcal J}}
\def\cK{{\mathcal K}}
\def\cL{{\mathcal L}}
\def\cM{{\mathcal M}}
\def\cN{{\mathcal N}}
\def\cO{{\mathcal O}}
\def\cP{{\mathcal P}}
\def\cQ{{\mathcal Q}}
\def\cR{{\mathcal R}}
\def\cS{{\mathcal S}}
\def\cT{{\mathcal T}}
\def\cU{{\mathcal U}}
\def\cV{{\mathcal V}}
\def\cW{{\mathcal W}}
\def\cX{{\mathcal X}}
\def\cY{{\mathcal Y}}
\def\cZ{{\mathcal Z}}

\def\dist#1{\left\|#1\right\|}
\def\nrp#1{\left\|#1\right\|_p}
\def\nrq#1{\left\|#1\right\|_m}
\def\nrqk#1{\left\|#1\right\|_{m_k}}
\def\Ln#1{\mbox{\rm {Ln}}\,#1}
\def\nd{\hspace{-1.2mm}}
\def\ord{{\mathrm{ord}}}
\def\Cc{{\mathrm C}}
\def\Pb{\,{\mathbf P}}

\def\va{{\mathbf{a}}}

\newcommand{\commF}[1]{\marginpar{%
\begin{color}{red}
\vskip-\baselineskip 
\raggedright\footnotesize
\itshape\hrule \smallskip F: #1\par\smallskip\hrule\end{color}}}

\newcommand{\commI}[1]{\marginpar{%
\begin{color}{blue}
\vskip-\baselineskip 
\raggedright\footnotesize
\itshape\hrule \smallskip I: #1\par\smallskip\hrule\end{color}}}

\newcommand{\commO}[1]{\marginpar{%
\begin{color}{cyan}
\vskip-\baselineskip 
\raggedright\footnotesize
\itshape\hrule \smallskip O: #1\par\smallskip\hrule\end{color}}}




\newcommand{\ignore}[1]{}

\def\vec#1{\boldsymbol{#1}}

\def\e{\mathbf{e}}



\def\GL{\mathrm{GL}}

\hyphenation{re-pub-lished}

\def\rank{{\mathrm{rk}\,}}
\def\dd{{\mathrm{dyndeg}\,}}
\def\lcm{{\mathrm{lcm}\,}}

\def\A{\mathbb{A}}
\def\B{\mathbf{B}}
\def \C{\mathbb{C}}
\def \F{\mathbb{F}}
\def \K{\mathbb{K}}
\def \L{\mathbb{L}}
\def \Z{\mathbb{Z}}
\def \P{\mathbb{P}}
\def \R{\mathbb{R}}
\def \Q{\mathbb{Q}}
\def \N{\mathbb{N}}
\def \U{\mathbb{U}}
\def \Z{\mathbb{Z}}

\def \nd{{\, | \hspace{-1.5 mm}/\,}}

\def\mand{\qquad\mbox{and}\qquad}

\def\Zn{\Z_n}

\def\Fp{\F_p}
\def\Fq{\F_q}
\def \fp{\Fp^*}
\def\\{\cr}
\def\({\left(}
\def\){\right)}
\def\fl#1{\left\lfloor#1\right\rfloor}
\def\rf#1{\left\lceil#1\right\rceil}
\def\vh{\mathbf{h}}
\def\ov#1{{\overline{#1}}}
\def\un#1{{\underline{#1}}}
\def\wh#1{{\widehat{#1}}}
\def\wt#1{{\widetilde{#1}}}
\newcommand{\abs}[1]{\left| #1 \right|}

\def\ZK{\Z_\K}
\def\LH{\cL_H}

\def \fI{\mathfrak{I}}
\def \fJ{\mathfrak{J}}
\def \fV{\mathfrak{V}}

\title[Lattice counting problem, II]
{On the error term of a lattice counting problem, II}

\author[O. Bordell\`{e}s]{Olivier Bordell\`{e}s}
\address{O.B.: 2 All\'ee de la combe, 43000 Aiguilhe, France}
\email{borde43@wanadoo.fr}

\keywords{Lattices, Weyl's bound for exponential sums, sums of fractional parts of polynomials.}

\subjclass[2010]{Primary 11L07, 11L15, 11P21}

\date{}

\begin{abstract} Under the Riemann Hypothesis, we improve the error term in the asymptotic formula related to the counting lattice problem studied in a first part of this work. The improvement comes from the use of Weyl's bound for exponential sums of polynomials and a device due to Popov allowing us to get an improved main term in the sums of certain fractional parts of polynomials.
\end{abstract}

\maketitle

\section{Introduction and result}

This work is the continuation of the paper \cite{borls} in which the following problem is studied. For integer $T \geqslant 1$, we let
$$
\cF(T) := \{a/b:(a,b)\in \Z^2, \ 0 \leqslant a < b \leqslant T, \ (a,b)=1\}
$$ 
be the set of Farey fractions. We  also define
$$
\cI(T) = \cF(T) \cap \left[ 0, \tfrac{1}{2} \right]
$$
and consider the quantity
$$
C(T) = \sum_{a/b \in \cI(T)} \#  \cC_{a,b}(T),
$$
where 
$$
\cC_{a,b}(T) :=  \cF(T) \cap \left [1-a^2/b^2,1 \right ].
 $$

As it is mentioned in \cite{borls}, this quantity $C(T)$ appears naturally in some counting problems for two-dimensional lattices and the main term of the asymptotic formula for $C(T)$ can be expressed via the cardinality 
$$
F(T) := \# \cF(T)
$$
of the set of Farey fractions and also second moment 
of the Farey fractions in $[0,\frac{1}{2}]$:
$$
G(T) := \sum_{\substack{\xi \in \cF(T)\\\xi \leqslant 1/2}} \xi^2.$$
More precisely, it is shown \cite[Theorem~1.1]{borls} that, unconditionally
$$C(T)  =  F(T)G(T) + O\left(T^3\delta(T^{1/2})\log T\right)$$
where $\delta(t)$ is the usual number-theoretic remainder function defined as 
$$ \delta(t) := \exp(-c(\log t)^{3/5} (\log\log t)^{-1/5}) \quad \left( c > 0 \right).$$
Under the Riemann Hypothesis, one can improve on the error term by using the well-known estimate \cite{bal}
\begin{equation}
   M(t) := \sum_{n \leqslant t} \mu(n) \ll t^{1/2} \rho(t) \label{Mertens:RH}
\end{equation}
where
$$\rho(t) := \exp \left( (\log t)^{1/2} (\log \log t)^{5/2+o(1)} \right)$$
and the authors derived in \cite[Theorem~1.4]{borls} the estimate
$$C(T)  =  F(T)G(T) + O\left(T^{752/283} \rho(T)  \log T \right)$$
under RH, with the help of Bourgain's exponent pair $(k, \ell) = \left ( \frac{13}{84}, \frac{55}{84} \right )$. Furthermore, it is pointed out that, if the \textit{exponent pair conjecture} is true, then the error-term may be sharpened to $O(T^{5/2 + o(1)})$. Note that
$$\frac{752}{283} \doteq \np{2.6572}.$$

The aim of this work is to show that there is no need to assume this very difficult conjecture in order to get this estimate. More precisely, we prove

\begin{theorem}
\label{thm:CT FTGT RH}
Assume the Riemann Hypothesis. Then
$$C(T)  =  F(T)G(T) + O\left(T^{5/2+o(1)}\right).$$
\end{theorem}

The idea is to estimate a sum of fractional parts using Weyl's bound for exponential sums of polynomials \cite{mon} and a device of Popov \cite{pop}, also used by Fomenko \cite{fom}, which allows us to improve the main term in sums of the shape
$$\sum_{N < n \leqslant 2N} \psi \left( P(n) \right)$$
where $P$ is any polynomial of degree $\geqslant 2$ and 
$$\psi(x) := x - \fl x  - \tfrac{1}{2}.$$

\section{Notation}

We take all the notation of \cite{borls} into account, in particular
$$E(T) := C(T) - F(T)G(T)$$
is the error term in the lattice counting problem considered here.
Let $\psi(x) := x - \fl x - \frac{1}{2}$ be the $1$st Bernoulli function and $\dist x  := \min \left( \frac{1}{2} - \psi(x) , \frac{1}{2} + \psi(x) \right)$ is the distance of $x$ to its nearest integer.

\medskip

For any $\beta \geqslant 0$, let $F_\beta$ be the multiplicative function defined by
$$F_\beta(n) :=  \sum_{d \mid n} \frac{\mu(d)^2}{d^\beta}= \prod_{p \mid n} \left( 1+ \frac{1}{p^\beta} \right).$$
Note that $F_0 = 2^\omega$ and
\begin{equation}
   \sum_{n \leqslant x} F_\beta(n) \ll \begin{cases} x \log x, & \textrm{if\ } \beta = 0 \\ x, & \textrm{if\ } \beta > 0. \end{cases} \label{fct:F}
\end{equation}

\medskip

Finally, if $f : \left( M,M+N \right]  \longrightarrow \R$ is any map and $\delta \in \left( 0,\frac{1}{4} \right]$, define
$$\mathcal{R} \left (f,N,\delta\right) := \#  \left \{ n \in \left( M,M+N \right] \cap \Z : \| f(n) \|  < \delta \right \}.$$

\section{Sums of fractional parts of polynomials}

This section is devoted to the proof of the following proposition, generalizing \cite[Theorems~1 and~2]{fom}, and which may have its own interest (see Remark~\ref{rem:VdC} below).

\begin{prop}
\label{prop:fracpol}
Let $q \in \Z_{\geqslant 1}$, $k \in \Z_{\geqslant 2}$, $N \in \Z_{\geqslant 1}$ large and $\alpha > 0$. Then, for any $\varepsilon > 0$
$$\sum_{\substack{N <n \leqslant 2N \\ (n,q)=1}} \psi \left( n^k \alpha \right) \ll_{k,\varepsilon}  \left( N \kappa \right)^\varepsilon \left( N \alpha^{2^{1-k}} F_{1-k2^{1-k}}(q) + N^{1-2^{1-k}} F_{1-2^{1-k}}(q) + \frac{N^{1-k2^{1-k}}F_1(q)}{\alpha^{2^{1-k}}} \right)$$
where $\kappa := \max \left( \alpha,\alpha^{-1} \right)$.
\end{prop}

For the proof, the following lemma is needed. This result is similar to \cite[(9)]{mon} but the statement does not need any rational approximation of $\alpha$ and the method of proof is quite different.

\begin{lemma}
\label{lem:sum}
Let $M \in \Z_{\geqslant 0}$, $N \in \Z_{\geqslant 1}$, $L \in \Z_{\geqslant 4}$ and $\alpha > 0$. Then
$$\sum_{M < n \leqslant M+N} \min \left( L , \frac{1}{\| n \alpha \|} \right) \ll LN \alpha + \left( N + \alpha^{-1} \right) \log L + L.$$
\end{lemma}

\begin{proof}
From \cite[Lemma~6.45]{bor}, we first have
$$\sum_{M < n \leqslant M+N} \min \left( L , \frac{1}{\| n \alpha \|} \right) \ll N + L \sum_{k=0}^{K-2} 2^{-k} \mathcal{R}\left (n \alpha,N,2^k L^{-1}\right)$$
where $K := \left \lfloor\frac{\log L}{\log 2} \right \rfloor$, and using \cite[Theorem~5.6]{bor} we get
\begin{eqnarray*}
   \sum_{M < n \leqslant M+N} \min \left( L , \frac{1}{\| n \alpha \|} \right) & \ll & N + L \sum_{k=0}^{K-2} 2^{-k} \left( N \alpha + 2^k N L^{-1} + 2^k (L \alpha)^{-1} + 1 \right) \\
   & \ll & N + LN \alpha + \left( N + \alpha^{-1} \right) K + L
\end{eqnarray*}
implying the asserted result.
\end{proof}

We now are in a position to prove Proposition~\ref{prop:fracpol}.

\begin{proof}[Proof of Proposition~\ref{prop:fracpol}]
One may assume $\alpha \in \left( 0,1 \right)$, otherwise $N \alpha^{2^{1-k}} \gg N$. For any $H \in \Z_{\geqslant 1}$, the left-hand side does not exceed
\begin{eqnarray*}
   & \ll & \frac{\varphi(q)}{q} \frac{N}{H} + \sum_{h \leqslant H} \frac{1}{h} \left | \sum_{\substack{N <n \leqslant 2N \\ (n,q)=1}} e \left( hn^k \alpha \right) \right | \\
   & \ll & \frac{N}{H} + \sum_{h \leqslant H} \frac{1}{h} \sum_{d \mid q} \mu(d)^2 \left | \sum_{\frac{N}{d} < n \leqslant \frac{2N}{d}} e \left( h n^k d^k \alpha \right) \right | \\
   & \ll & \frac{N}{H} + \left( \sum_{\substack{d \mid q \\ d > \frac{1}{4}N}} + \sum_{\substack{d \mid q \\ d \leqslant \frac{1}{4}N}} \right) \mu(d)^2 \sum_{h \leqslant H} \frac{1}{h} \left| \sum_{\frac{N}{d} < n \leqslant \frac{2N}{d}} e \left( h n^k d^k \alpha \right) \right | \\
   & \ll & \frac{N}{H} + 2^{\omega(q)} \log H + \sum_{\substack{d \mid q \\ d \leqslant \frac{1}{4}N}} \mu(d)^2 \sum_{h \leqslant H} \frac{1}{h} \left| \sum_{\frac{N}{d} < n \leqslant \frac{2N}{d}} e \left( h n^k d^k \alpha \right) \right |
\end{eqnarray*}
and using Weyl's bound \cite[(4) p. 40]{mon} we get
\begin{eqnarray*}
   \left | \sum_{\substack{N <n \leqslant 2N \\ (n,q)=1}} \psi \left( n^k \alpha \right) \right | & \ll & \frac{N}{H} + 2^{\omega(q)} \log H + \sum_{\substack{d \mid q \\ d \leqslant \frac{1}{4}N}} \mu(d)^2 \sum_{h \leqslant H} \frac{1}{h} \left\lbrace \left( \frac{N}{d} \right)^{2^{k-1}-1} \right. \\
   & & \left. {} + \left( \frac{N}{d} \right)^{2^{k-1}-k + \varepsilon}\sum_{\ell \leqslant k!(N/d)^{k-1}} \min \left( \frac{N}{d} , \frac{1}{\| \ell h  d^k \alpha \|} \right) \right\rbrace^{2^{1-k}} \\
   & \ll & \frac{N}{H} + \left( N^{1-2^{1-k}} F_{1-2^{1-k}}(q) + 2^{\omega(q)} \right) \log H \\
   & & + N^{1-k2^{1-k} + \varepsilon} \sum_{\substack{d \mid q \\ d \leqslant \frac{1}{4}N}} \frac{\mu(d)^2}{d^{1-k2^{1-k}}} \, S_{H,N}(d)
\end{eqnarray*}
where
$$S_{H,N}(d) := \sum_{h \leqslant H} \frac{1}{h} \left( \sum_{\ell \leqslant k!(N/d)^{k-1}} \min \left( \frac{N}{d} , \frac{1}{\| \ell h  d^k \alpha \|} \right) \right)^{2^{1-k}}.$$
Assume $d \leqslant \frac{1}{4}N$. H\"{o}lder's inequality with $\lambda = \frac{2^k}{2^k-2}$ yields
\begin{eqnarray*}
   S_{H,N}(d) & \leqslant & \left( \sum_{h \leqslant H} \frac{1}{h} \right)^{1-2^{1-k}} \left( \sum_{h \leqslant H} \frac{1}{h} \sum_{\ell \leqslant k!(N/d)^{k-1}} \min \left( \frac{N}{d} , \frac{1}{\| \ell h  d^k \alpha \|} \right) \right)^{2^{1-k}} \\
   & \leqslant & ( \log eH)^{1-2^{1-k}} \left( \sum_{n \leqslant k!H(N/d)^{k-1}} \min \left( \frac{N}{d} , \frac{1}{\| n  d^k \alpha \|} \right) \sum_{\substack{h \mid n \\ h \leqslant H \\ n/h \leqslant k!(N/d)^{k-1}}} \frac{1}{h}\right)^{2^{1-k}} \\
   &=& ( \log eH)^{1-2^{1-k}} \left( \sum_{j=0}^{k!H-1} \sum_{j(N/d)^{k-1} < n \leqslant (j+1)(N/d)^{k-1}} \min \left( \frac{N}{d} , \frac{1}{\| n  d^k \alpha \|} \right) \sum_{\substack{h \mid n \\ h \leqslant H \\ n/h \leqslant k!(N/d)^{k-1}}} \frac{1}{h}\right)^{2^{1-k}}.
\end{eqnarray*}
Following \cite[(13),(14)]{pop} (see also \cite{fom}), note that, in the innersum
$$\frac{1}{h} \leqslant \frac{k! N^{k-1}}{nd^{k-1}} < \frac{k!}{j} \quad \left( j \geqslant 1 \right)$$
so that
\begin{eqnarray*}
   S_{H,N}(d) & \leqslant & ( \log eH)^{1-2^{1-k}} \left( \sum_{n \leqslant (N/d)^{k-1}} \frac{\sigma(n)}{n} \min \left( \frac{N}{d} , \frac{1}{\| n  d^k \alpha \|} \right) \right. \\
   & & \left. {} + \sum_{j=1}^{k!H-1} \frac{k!}{j} \sum_{j(N/d)^{k-1} < n \leqslant (j+1)(N/d)^{k-1}} \tau(n) \min \left( \frac{N}{d} , \frac{1}{\| n  d^k \alpha \|} \right) \right)^{2^{1-k}}
\end{eqnarray*}
and the crude bounds $\tau(n) \ll_\varepsilon n^\epsilon$ and $\sigma(n) \ll_\varepsilon n^{1+\epsilon}$, along with Lemma~\eqref{lem:sum} used with $M = j(N/d)^{k-1}$, $N$ replaced by $(N/d)^{k-1}$, $\alpha$ replaced by $d^k \alpha$ and $L=\frac{N}{d} \geqslant 4$, yield
\begin{eqnarray*}
   (NH)^{-\varepsilon} S_{H,N}(d) & \ll & \left\lbrace \left( N^k \alpha + \left( \frac{N}{d} \right)^{k-1} + d^{-k}\alpha^{-1} \right) \left( \sum_{j=1}^{k!H-1} \frac{1}{j} + 1 \right) \right\rbrace^{2^{1-k}} \\
   & \ll & N^{k 2^{1-k}} \alpha^{2^{1-k}} + \left( \frac{N}{d} \right)^{(k-1)2^{1-k}}+ d^{-k 2^{1-k}}\alpha^{-2^{1-k}}.
\end{eqnarray*}
Consequently
\begin{eqnarray*}
   \left | \sum_{\substack{N <n \leqslant 2N \\ (n,q)=1}} \psi \left( n^k \alpha \right) \right | & \ll & \frac{N}{H} + \left( N^{1-2^{1-k}} F_{1-2^{1-k}}(q) + 2^{\omega(q)} \right) \log H \\
   & & {} + H^{\varepsilon} N^{1-k2^{1-k} + \varepsilon} \sum_{\substack{d \mid q \\ d \leqslant \frac{1}{4}N}} \frac{\mu(d)^2}{d^{1-k2^{1-k}}} \left( N^{k 2^{1-k}} \alpha^{2^{1-k}} + \left( \frac{N}{d} \right)^{(k-1)2^{1-k}}+ d^{-k 2^{1-k}}\alpha^{-2^{1-k}} \right) \\
   & \ll & \frac{N}{H} + \left( N^{1-2^{1-k}} F_{1-2^{1-k}}(q) + 2^{\omega(q)}\right) \log H \\
   & & {} + (NH)^{\varepsilon} \left( N \alpha^{2^{1-k}} F_{1-k2^{1-k}}(q) + N^{1-2^{1-k}} F_{1-2^{1-k}}(q) + \frac{N^{1-k2^{1-k}}F_1(q)}{\alpha^{2^{1-k}}} \right)
\end{eqnarray*}
and the choice of $H = \left \lfloor \alpha^{-2^{1-k}} \right \rfloor$ allows us to achieve the proof.
\end{proof}

\begin{rem}
\label{rem:VdC}
With $q=1$, Proposition~\ref{prop:fracpol} yields
$$(N\kappa)^{-\varepsilon} \sum_{N <n \leqslant 2N} \psi \left( n^k \alpha \right) \ll_{k,\varepsilon} N \alpha^{2^{1-k}} + N^{1-2^{1-k}} + N^{1-k2^{1-k}}\alpha^{-2^{1-k}}$$
whereas Van der Corput's method \cite[Theorem~2.8]{GrKol} provides 
$$\sum_{N <n \leqslant 2N} \psi \left( n^k \alpha \right) \ll_k N \alpha^{1/(2^k-1)} + N^{1-2^{1-k}} + N^{1 - 2^{1-k}-2^{4-2k}} \alpha^{-2^{1-k}}$$
so that, for any $k \in \Z_{\geqslant 2}$ and $\alpha \in \left( 0,1 \right)$, Proposition~\ref{prop:fracpol} improves significantly the first term, sometimes called the main term, and the secondary terms are of the same strength.
\end{rem}

\section{Proof of Theorem~\ref{thm:CT FTGT RH}}

From \cite[(2.2)]{borls} we get

\begin{eqnarray*}
   E(T) &=& - \sum_{a/b \in \cI(T)} \sum_{d \leqslant T} M \left( \frac{T}{d} \right) \psi \left( \frac{da^2}{b^2} \right ) - \frac{1}{2} \sum_{a/b \in I(T)} \sum_{d \leqslant T} M \left( \frac{T}{d} \right) \\
   &=& - \sum_{a/b \in \cI(T)} \sum_{d \leqslant T} M \left( \frac{T}{d} \right) \psi \left( \frac{da^2}{b^2} \right ) + O \left( T^2 \right) := \Sigma(T) + O \left( T^2 \right)
\end{eqnarray*}
say. Now from \eqref{Mertens:RH}

\begin{eqnarray*}
    \left | \Sigma(T) \right | & = & \left | \sum_{d \leqslant T} M \left( \frac{T}{d} \right) \sum_{b \leqslant T} \sum_{\substack{a \leqslant \frac{1}{2}b \\ (a,b)=1}} \psi \left( \frac{da^2}{b^2} \right ) \right | \\
   & \ll & T^{1/2} \rho(T) \sum_{d \leqslant T} \frac{1}{d^{1/2}} \sum_{b \leqslant T} \left | \sum_{\substack{a \leqslant \frac{1}{2}b \\ (a,b)=1}} \psi \left( \frac{da^2}{b^2} \right ) \right | \\
   & \ll & T^{1/2} \rho(T) \sum_{d \leqslant T} \frac{1}{d^{1/2}} \, \sum_{b \leqslant T} \, \max_{A \leqslant \frac{1}{2}b} \left | \sum_{\substack{A < a \leqslant 2A \\ (a,b)=1}} \psi \left( \frac{da^2}{b^2} \right ) \right | \log b.
\end{eqnarray*}

We use Proposition~\ref{prop:fracpol} with $k=2$, i.e.
$$\sum_{\substack{N <n \leqslant 2N \\ (n,q)=1}} \psi \left( n^2 \alpha \right) \ll_{k,\varepsilon}  \left( N \kappa \right)^\varepsilon \left( N \alpha^{1/2} 2^{\omega(q)} + N^{1/2} F_{1/2}(q) + \frac{F_1(q)}{\alpha^{1/2}} \right)$$
with $N=A$, $q=b$ and $\alpha = db^{-2}$, yielding

\begin{eqnarray*}
   \left | \Sigma(T) \right | & \ll & T^{1/2 + o(1)} \sum_{d \leqslant T} \frac{1}{d^{1/2}} \sum_{b \leqslant T} \left( d^{1/2} 2^{\omega(b)} + b^{1/2} F_{1/2}(b)  + \frac{b F_1(b)}{d^{1/2}} \right)  \\
   & \ll & T^{1/2 + o(1)} \sum_{d \leqslant T} \frac{1}{d^{1/2}} \left( Td^{1/2} + T^{3/2} + T^2 d^{-1/2} \right)  \\
   & \ll & T^{5/2+o(1)}
\end{eqnarray*}
where we used the bound \eqref{fct:F} in the penultimate line. This completes the proof.
\qed

\section*{Acknowledgements}
The author deeply thanks Profs.~Florian Luca and~Igor Shparlinski for their precious advice and unfailing support.

\end{document}